\RequirePackage{fix-cm}
\documentclass[12pt,english]{extarticle}
\usepackage[T1]{fontenc}
\usepackage[utf8]{inputenc}
\usepackage{xcolor}
\usepackage{babel}
\usepackage{mathrsfs}
\usepackage{mathtools}
\usepackage{amsmath}
\usepackage{amsthm}
\usepackage{amssymb}
\usepackage{cancel}
\usepackage{geometry}
\usepackage[pdfusetitle,
 bookmarks=true,bookmarksnumbered=false,bookmarksopen=false,
 breaklinks=false,pdfborder={0 0 1},backref=false,colorlinks=true]
 {hyperref}
\hypersetup{
 linkcolor=blue,urlcolor=blue,citecolor=blue}

\makeatletter
\@ifundefined{date}{}{\date{}}
\usepackage{babel}
\usepackage{amsthm}
\usepackage{graphicx}
\usepackage{mlmodern}

\usepackage{xcolor}

\input{glyphtounicode}
\usepackage{fix-cm}
\usepackage{color}
\usepackage{bm}
\usepackage{mathrsfs}
\usepackage{graphicx}
\usepackage{booktabs}

\AtBeginDocument{

}

\theoremstyle{plain}
\newtheorem{thm}{\protect\theoremname}\theoremstyle{definition}
\newtheorem{defn}[thm]{\protect\definitionname}\theoremstyle{remark}
\newtheorem{rem}[thm]{\protect\remarkname}\theoremstyle{plain}
\newtheorem{prop}[thm]{\protect\propositionname}\providecommand{\corollaryname}{Corollary}
\providecommand{\definitionname}{Definition}
\providecommand{\examplename}{Example}
\providecommand{\propositionname}{Proposition}
\providecommand{\remarkname}{Remark}
\providecommand{\theoremname}{Theorem}

\makeatother

\theoremstyle{definition}
\newtheorem{example}[thm]{\protect\examplename}
\providecommand{\definitionname}{Definition}
\providecommand{\examplename}{Example}
\providecommand{\propositionname}{Proposition}
\providecommand{\remarkname}{Remark}
\providecommand{\theoremname}{Theorem}

\begin{document}
\title{A seminorm-only characterization of analytic Besov spaces on the disc}
\author{Maher Boudabra \thanks{Monastir Preparatory Engineering Institute, Monastir University, Tunisia.}}
\maketitle
\begin{abstract}
We introduce the space $\mathcal{W}^{s,p}(\mathbb{D})$ of analytic
functions $u$ on the unit disc such that the radial restrictions
$u_{r}(\xi):=u(r\xi)$ satisfy the Gagliardo seminorm-only bound
\[
\sup_{0<r<1}[u_{r}]_{W^{s,p}(\mathbb{S}^{1})}<\infty,
\]
with no $\emph{a priori}$ control of $\sup_{r}\|u_{r}\|_{L^{p}(\mathbb{S}^{1})}$. Our main result shows that this assumption already forces $u\in H^{p}(\mathbb{D})$
and that the radial boundary trace $u^{*}$ belongs to $W^{s,p}(\mathbb{S}^{1})$,
with $u_{r}\to u^{*}$ in $W^{s,p}(\mathbb{S}^{1})$ as $r\to1^{-}$.
The key mechanism combines the mean-value property (which pins the
constant mode at $u(0)$) with a fractional Poincar$\'e$ inequality
on $\mathbb{S}^{1}$, recovering $L^{p}$ control from oscillation
alone. As a consequence, the trace map $u\mapsto u^{*}$ is a surjective
isomorphism $\mathcal{W}^{s,p}(\mathbb{D})\xrightarrow{\sim}B^{s}_{p,p,+}(\mathbb{S}^{1})$
with explicit norm equivalence.
\end{abstract}

\section{Introduction}

Let $\Omega$ be a measurable set of $\mathbb{R}^{N}$, $N\in\mathbb{N}-\{0\}$.
For $u\in L^{p}(\Omega)$ the Gagliardo seminorm of $u$ is defined
by the quantity

\[
[u]_{W^{s,p}(\Omega)}:=\begin{cases}
{\displaystyle \bigg(\int_{\Omega}\int_{\Omega}\frac{|u(x)-u(y)|^{p}}{|x-y|^{N+sp}}\;dx\,dy\bigg)^{\frac{1}{p}},} & 1\le p<\infty,\\[1.2em]
{\displaystyle \text{ess}\sup_{x,y\in\Omega,\ x\neq y}\frac{|u(x)-u(y)|}{|x-y|^{s}},} & p=+\infty.
\end{cases}
\]
The quantity $[\cdot]_{W^{s,p}(\Omega)}$ is a seminorm. It assigns
the magnitude zero to constant functions for example. 
\begin{defn}
The $\emph{fractional Sobolev space}$ $W^{s,p}(\Omega)$ is defined
by 
\[
W^{s,p}(\Omega):=\Big\{ u\in L^{p}(\Omega):[u]_{W^{s,p}(\Omega)}<+\infty\Big\},
\]
endowed with the norm 
\[
\|u\|_{W^{s,p}(\Omega)}:=\|u\|_{L^{p}(\Omega)}+[u]_{W^{s,p}(\Omega)}.
\]
For $p=2$ we write $H^{s}(\Omega):=W^{s,2}(\Omega)$. 
\end{defn}

\begin{rem}[Fourier characterization on $\mathbb{R}^{N}$]
Note that 
\[
\|u\|_{L^{p}(\Omega)}\leq\|u\|_{W^{s,p}(\Omega)}
\]
Hence 
\[
W^{s,p}(\Omega)\hookrightarrow L^{p}(\Omega).
\]
When $\Omega=\mathbb{R}^{N}$ and $p=2$, the space $H^{s}(\mathbb{R}^{N})$
admits the equivalent description 
\[
H^{s}(\mathbb{R}^{N})=\Big\{ u\in L^{2}(\mathbb{R}^{N}):(1+|\xi|^{2})^{\frac{s}{2}}\,\widehat{u}(\xi)\in L^{2}(\mathbb{R}^{N})\Big\},
\]
with equivalent norm 
\[
\|u\|^{2}_{H^{s}(\mathbb{R}^{N})}\propto\int_{\mathbb{R}^{N}}(1+|\xi|^{2})^{s}|\widehat{u}(\xi)|^{2}\,d\xi.
\]
Moreover, 
\[
[u]^{2}_{H^{s}(\mathbb{R}^{N})}:=\int_{\mathbb{R}^{N}}\int_{\mathbb{R}^{N}}\frac{|u(x)-u(y)|^{2}}{|x-y|^{N+2s}}\,dx\,dy\propto\int_{\mathbb{R}^{N}}|\xi|^{2s}|\widehat{u}(\xi)|^{2}\,d\xi.
\]
In fact, $H^{s}(\mathbb{R}^{N})$ is a Hilbert space where the inner
product is given by 
\[
(u,v)_{H^{s}(\mathbb{R}^{N})}=\int_{\mathbb{R}^{N}}\int_{\mathbb{R}^{N}}\frac{(u(x)-u(y))(v(x)-v(y))}{|x-y|^{N+2s}}\,dx\,dy
\]
Thus, in the Hilbertian case $H^{s}(\mathbb{R}^{N})$ may be interpreted
as the space of functions having ``$s$ derivatives in $L^{2}$''
in a Fourier sense. The double integral in the Gagliardo seminorm
measures the average size of the difference quotient 
\[
\frac{\vert u(x)-u(y)\vert^{p}}{|x-y|^{N+sp}}
\]
over all pairs $(x,y)\in\Omega\times\Omega$, where the kernel $|x-y|^{-N-sp}$
plays the role of a ``fractional Jacobian''. In this sense $W^{s,p}(\Omega)$
controls both the size of $u$ and its oscillations on all scales,
interpolating between $L^{p}(\Omega)$ (when $s=0$) and the classical
Sobolev space $W^{1,p}(\Omega)$ (when $s=1$). Fractional Sobolev
spaces are naturally connected with the fractional Laplacian, which
is the nonlocal counterpart of the classical Laplace operator $\Delta$.
The fractional Laplacian $(-\Delta)^{s}$ is defined as the Fourier
multiplier 
\[
\widehat{(-\Delta)^{s}u}(\xi)=|\xi|^{2s}\widehat{u}(\xi).
\]
In particular, weak solutions of nonlocal equations driven by $(-\Delta)^{s}$
arise naturally as critical points of energies defined on $H^{s}(\mathbb{R}^{N})$,
making fractional Sobolev spaces the canonical functional setting
for the analysis of such problems. 
\end{rem}

In our setting, the (periodic) fractional Sobolev space is 
\[
W^{s,p}(\mathbb{S}^{1})=\{\gamma\in L^{p}(\mathbb{S}^{1},\mathbb{R})\mid[\gamma]_{W^{s,p}}<+\infty\}
\]
where 
\[
\begin{alignedat}{1}[\gamma]_{W^{s,p}(\mathbb{S}^{1})} & :=\left(\iint_{(-\pi,\pi)\times(-\pi,\pi)}{\textstyle \frac{\vert\gamma(e^{i\xi})-\gamma(e^{i\sigma})\vert^{p}}{\vert e^{i\xi}-e^{i\sigma}\vert^{1+sp}}}d\xi\,d\sigma\right)^{\frac{1}{p}}\\
 & =\left(\frac{1}{2^{1+sp}}\iint_{(-\pi,\pi)\times(-\pi,\pi)}{\textstyle \frac{\vert\gamma(e^{i\xi})-\gamma(e^{i\sigma})\vert^{p}}{\vert\sin(\frac{\xi-\sigma}{2})\vert^{1+sp}}}d\xi\,d\sigma\right)^{\frac{1}{p}}.
\end{alignedat}
\]

More generally, if 
\[
u(\theta)=\sum_{k\in\mathbb{Z}}\hat{u}(k)\,e^{ik\theta}\in\mathcal{D}'(\mathbb{S}^{1})
\]
(periodic distribution) then 
\[
(-\Delta)^{s}u(\theta)=\sum_{k\in\mathbb{Z}}\vert k\vert^{2s}\hat{u}(k)\,e^{ik\theta}
\]
with the convention $(-\Delta)^{0}u=u$. In particular 
\[
\Vert(-\Delta)^{\frac{s}{2}}u\Vert^{2}_{L^{2}}=\sum_{k\in\mathbb{Z}}\vert k\vert^{2s}\vert\hat{u}(k)\vert^{2}.
\]
The reader may notice that the kernel 
\[
\varrho_{s}(\theta)=\frac{1}{2^{1+sp}}\frac{1}{\vert\sin(\theta)\vert^{1+sp}}
\]
can be viewed as the wrapping (periodization) of the one dimensional
kernel $\frac{1}{\vert x\vert^{1+sp}}$. In particular, the following
estimate holds. 
\begin{prop}
The Gagliardo seminorm $[u]_{W^{s,2}}$ satisfies 
\[
[u]^{2}_{W^{s,2}}\propto\Vert(-\Delta)^{\frac{s}{2}}u\Vert^{2}_{L^{2}},
\]
with explicit Fourier identity 
\begin{equation}
[u]^{2}_{W^{s,2}(\mathbb{S}^{1})}:=\frac{1}{2^{1+2s}}\iint_{(-\pi,\pi)\times(-\pi,\pi)}{\textstyle \frac{|u(\xi)-u(\sigma)|^{2}}{\bigl|\sin\frac{\xi-\sigma}{2}\bigr|^{1+2s}}}\,d\xi\,d\sigma=\eta_{s}\sum_{k\in\mathbb{Z}}|k|^{2s}|\hat{u}(k)|^{2}.\label{eq:Hs-Fourier}
\end{equation}
\end{prop}

For $0<r<1$, the Poisson kernel is defined by 
\[
P_{r}(\theta)=\frac{1-r^{2}}{1-2r\cos\theta+r^{2}},\qquad\theta\in[-\pi,\pi],
\]
and satisfies 
\[
\int^{\pi}_{-\pi}P_{r}(\theta)\,\frac{d\theta}{2\pi}=1,\qquad P_{r}(\theta)\ge0.
\]
For a distribution (or function) $u$ on $\mathbb{S}^{1}$ we define
the Poisson mean 
\[
\mathscr{P}_{r}u:=P_{r}*u=\int_{\mathbb{S}^{1}}P_{r}(t)u(\cdot-t)\,\frac{dt}{2\pi}.
\]
That is, $\mathscr{P}_{r}u$ is a harmonic function in the unit disc,
with boundary value $u$ whenever $u\in L^{1}(\mathbb{S}^{1})$ (see
Rudin's book for example). 
\begin{defn}
Let $f$ be an analytic function acting on the unit disc and $p>0$.
The $p^{th}$- Hardy norm of $f$ is defined by 
\begin{equation}
\Vert f\Vert_{\mathbf{H}^{p}(\mathbb{D})}:=\sup_{0\leq r<1}\left\{ \frac{1}{2\pi}\int^{2\pi}_{0}|f(re^{\theta i})|^{p}d\theta\right\} ^{\frac{1}{p}}.\label{hardy}
\end{equation}
\end{defn}

The Hardy space (of index $p$), denoted by $\mathbf{H}^{p}(\mathbb{D})$
is the space of functions with finite $p^{th}$- Hardy norm. A crucial
result about Hardy norms is that, if $\Vert f\Vert_{\mathbf{H}^{p}(\mathbb{D})}$
is finite then $f$ has a radial extension to the boundary 
\[
f^{*}(\xi):=\lim_{r\rightarrow1}f(r\xi)
\]
exists a.e for all $\xi\in\mathbb{S}^{1}$. Moreover 
\[
\Vert f^{*}\Vert_{p}=\Vert f\Vert_{\mathbf{H}^{p}(\mathbb{D})}.
\]
We invite the reader to look at \cite{Rudin2001,duren2000theory}
for a more in-depth comprehensive treatment of the theory of Hardy
spaces. For the sake of simplicity, we shall keep writing $f$ instead
of $f^{*}$. The Hardy space is also defined for harmonic functions,
and often called real Hardy space (See \cite{duren2000theory}).

The classical Hardy-Besov spaces $B^{s}_{p,p,+}$ of analytic functions
on the disc are defined as the analytic subspace of $B^{s}_{p,p}(\mathbb{S}^{1})\simeq W^{s,p}(\mathbb{S}^{1})$.
They have been extensively studied, notably by Flett \cite{Flett1972},
Oswald \cite{Oswald1983}, and Triebel \cite{TriebelChar1988}. In
these works, boundary smoothness is either prescribed from the outset
or built into the definition via Hardy norms. The present paper takes
a different starting point: we ask how much can be deduced from the
sole knowledge that the Gagliardo seminorms $[u_{r}]_{W^{s,p}(\mathbb{S}^{1})}$
are uniformly bounded, with no $L^{p}$ or Hardy assumption. The answer,
given by our main theorem below, is that this minimal hypothesis already
forces full membership in $\mathbf{H}^{p}(\mathbb{D})$ and $W^{s,p}$
regularity of the boundary trace. The key tool is the interplay between
the mean-value property of analytic functions and the fractional Poincaré
inequality on $\mathbb{S}^{1}$.

\medskip
\noindent\textbf{Main result and contribution.}
Given an analytic function $u$ on $\mathbb{D}$, write its radial restrictions
$u_r(\xi):=u(r\xi)$ for $0<r<1$ and $\xi\in\mathbb{S}^1$.
We introduce the Banach space
\[
\mathcal{W}^{s,p}(\mathbb{D})
:=\left\{u \text{ analytic on }\mathbb{D}:\ \sup_{0<r<1}[u_r]_{W^{s,p}(\mathbb{S}^1)}<\infty\right\},
\qquad
\|u\|_{\mathcal{W}^{s,p}}:=|u(0)|+\sup_{0<r<1}[u_r]_{W^{s,p}(\mathbb{S}^1)}.
\]
The distinctive feature is that the defining assumption controls only tangential oscillation of each slice
and makes no \emph{a priori} assumption on $\sup_r\|u_r\|_{L^p(\mathbb{S}^1)}$.
Theorem \ref{Thm main result} shows that, for analytic functions, this seminorm-only hypothesis already forces
$u\in \mathbf{H}^p(\mathbb{D})$ and yields a boundary trace $u^*\in W^{s,p}(\mathbb{S}^1)$ with
$u_r\to u^*$ in $W^{s,p}(\mathbb{S}^1)$ as $r\to1^{-}$.
The key mechanism is the combination of the mean-value property $\overline{u_r}=u(0)$ (which pins the constant mode)
with a fractional Poincar\'e inequality on $\mathbb{S}^1$, recovering $L^p$ control from oscillation alone.
As a consequence, the trace map $u\mapsto u^*$ is identified as a surjective isomorphism
$\mathcal{W}^{s,p}(\mathbb{D})\xrightarrow{\sim}B^s_{p,p,+}(\mathbb{S}^1)$ with explicit norm equivalence constants
(Theorem~\ref{thm:besov-iso}). This identification is a \emph{consequence} of the seminorm criterion and is not used to
prove Theorem \ref{Thm main result}.

\medskip
\noindent\textbf{Organization.}
Section~\ref{sec:space} studies $\mathcal{W}^{s,p}(\mathbb{D})$ as a Banach space and proves the trace isomorphism
with $B^s_{p,p,+}(\mathbb{S}^1)$. Section~\ref{sec:applications} contains applications (harmonic and semilinear problems,
a variational characterization, and a stochastic interpretation), and Section~\ref{sec:examples} collects explicit examples.
\medskip
We first recall the strong continuity of the Poisson semigroup on $W^{s,p}(\mathbb S^1)$,
which is used in Step~3 of the proof of Theorem~\ref{Thm main result}.
\begin{thm}[{\cite{Bui1984,Kalyabin1988,TriebelChar1988}}]
\label{fractional Poisson} Let $1<p<\infty$ and $0<s<1$. Then
for every $u\in W^{s,p}(\mathbb{S}^{1})$, 
\[
\|\mathscr{P}_{r}u-u\|_{W^{s,p}(\mathbb{S}^{1})}\longrightarrow0\qquad\text{as }r\to1^{-}.
\]
(Here we use the identification $W^{s,p}(\mathbb{S}^{1})\simeq B^{s}_{p,p}(\mathbb{S}^{1})$;
the cited references establish strong continuity in $B^{s}_{p,p}$.) 
\end{thm}

Theorem \ref{fractional Poisson} is referred to as the continuity
of the fractional Poisson operator. It says that one can construct
a harmonic function inside the unit disc from boundary data in $W^{s,p}(\mathbb{S}^{1})$,
and that the construction is continuous. The natural converse question
is the following: \emph{given an analytic function $u$ in the unit
disc whose radial restrictions $u_{r}$ have uniformly controlled
fractional oscillation, must $u$ admit a boundary trace with the
same regularity?}

In the classical Hardy space theory, membership in $\mathbf{H}^{p}(\mathbb{D})$
is established by controlling the full $L^{p}$ means $\sup_{r}\|u_{r}\|_{L^{p}}$.
The distinctive feature of the following theorem is that \emph{no
such $L^{p}$ control is assumed}. Only the homogeneous Gagliardo
seminorm $[u_{r}]_{W^{s,p}(\mathbb{S}^{1})}$, which measures oscillation
but discards mean values, is required to be uniformly bounded. The
$L^{p}$ control, and hence Hardy space membership, is then recovered
as a consequence, via the mean-value property of analytic functions
and a fractional Poincaré inequality on $\mathbb{S}^{1}$. This recovery
of size from oscillation is the core contribution of the result.

By abuse of language, we say that a complex-valued function is harmonic
in the unit disc if both its real and imaginary parts are harmonic.
In what follows, $0<s<1$ and $1<p<+\infty$ unless otherwise mentioned.
\begin{thm}[Seminorm criterion for Hardy membership and boundary trace]
\label{Thm main result} Let $u$ be an analytic function in the
unit disc. Since $u$ is analytic, each radial restriction $u_{r}:\xi\mapsto u(r\xi)$
is smooth on $\mathbb{S}^{1}$ and in particular lies in $W^{s,p}(\mathbb{S}^{1})$
for every $r\in(0,1)$. Assume that the Gagliardo seminorms of the
radial restrictions are uniformly bounded, i.e. 
\begin{equation}
\sup_{0<r<1}[u_{r}]_{W^{s,p}(\mathbb{S}^{1})}<+\infty.\label{eq:uniform-gagliardo}
\end{equation}
Then the following conclusions hold. 
\begin{enumerate}
\item $u$ belongs to the Hardy space $\mathbf{H}^{p}(\mathbb{D})$, and
in particular admits a radial boundary trace $u^{*}\in L^{p}(\mathbb{S}^{1})$
satisfying 
\[
\lim_{r\to1^{-}}u(re^{i\theta})=u^{*}(e^{i\theta})\quad\text{for almost every }\theta\in(-\pi,\pi),
\]
with convergence also in $L^{p}(\mathbb{S}^{1})$. 
\item The boundary trace has fractional Sobolev regularity: $u^{*}\in W^{s,p}(\mathbb{S}^{1})$. 
\item The Gagliardo seminorms converge: 
\[
\lim_{r\to1^{-}}[u_{r}-u^{*}]_{W^{s,p}(\mathbb{S}^{1})}=0,
\]
and in particular $[u_{r}]_{W^{s,p}(\mathbb{S}^{1})}\to[u^{*}]_{W^{s,p}(\mathbb{S}^{1})}$
as $r\to1^{-}$. 
\end{enumerate}
\end{thm}

\begin{rem}
The hypothesis \eqref{eq:uniform-gagliardo} is \emph{a priori} weaker
than requiring $u\in\mathbf{H}^{p}(\mathbb{D})$ with $u^{*}\in W^{s,p}(\mathbb{S}^{1})$
from the outset: the latter implies the former by Poisson contraction
(Proposition~\ref{prop:poisson-contraction}), since 
\[
[u_{r}]_{W^{s,p}}=[(P_{r}*u^{*})]_{W^{s,p}}\le[u^{*}]_{W^{s,p}}
\]
 uniformly in $r$, but the former does not trivially imply the latter.
The content of the theorem is precisely that these two conditions
are in fact equivalent, and that the seminorm-only side is sufficient. 
\end{rem}

The proof proceeds in three steps: 
\begin{enumerate}
\item \textbf{Hardy membership.} We show $u\in\mathbf{H}^{p}(\mathbb{D})$
using the fractional Poincaré inequality on $\mathbb{S}^{1}$ together
with the mean-value property $\overline{u_{r}}=u(0)$. This is the
key step: the constant mode is pinned by $u(0)$, allowing the seminorm
bound to control the full $L^{p}$ norm of $u_{r}$. 
\item \textbf{Boundary regularity.} Once $u\in\mathbf{H}^{p}(\mathbb{D})$,
the boundary trace $u^{*}\in L^{p}(\mathbb{S}^{1})$ exists. We promote
this to $u^{*}\in W^{s,p}(\mathbb{S}^{1})$ using reflexivity of $W^{s,p}$
and a weak compactness argument. 
\item \textbf{Seminorm convergence.} We use the strong continuity of the
Poisson semigroup on 
\[
B^{s}_{p,p}(\mathbb{S}^{1})\simeq W^{s,p}(\mathbb{S}^{1})
\]
 to conclude that 
\[
[u_{r}-u^{*}]_{W^{s,p}}\to0.
\]
\end{enumerate}

\subsubsection*{Step 1: Case $p=2$ (Hilbertian argument)}

\noindent The case $p=2$ is presented separately because the Fourier
series structure of $H^{2}(\mathbb{D})$ makes the argument completely
explicit and self-contained, with no appeal to reflexivity or weak
compactness. The general case below subsumes this one, but the $p=2$
argument serves as transparent motivation for all three steps.

\noindent For each $0<r<1$, write the Taylor expansion of $u$ as
\[
u_{r}(e^{\theta i})=\sum^{+\infty}_{k=0}a_{k}r^{k}e^{ik\theta},\qquad\theta\in(-\pi,\pi),
\]
In particular, the Fourier coefficients of $u_{r}$ on $\mathbb{S}^{1}$
are 
\[
(u_{r})_{k}=a_{k}r^{k}\quad\text{for }k\ge0,\qquad(u_{r})_{k}=0\quad\text{for }k<0.
\]
In particular, by \ref{eq:Hs-Fourier} we obtain 
\begin{align*}
[u_{r}]^{2}_{W^{s,2}(\mathbb{S}^{1})} & \propto\sum_{k\in\mathbb{Z}}|k|^{2s}\,|(u_{r})_{k}|^{2}\\
 & =\underset{=S(r)}{\underbrace{\sum^{\infty}_{k=1}k^{2s}|a_{k}|^{2}r^{2k}}}.
\end{align*}
By hypothesis \ref{eq:uniform-gagliardo} with $p=2$, we have 
\[
\sup_{0<r<1}S(r)<+\infty.
\]
Observe that for each fixed $k\ge1$ the function $r\mapsto r^{2k}$
is increasing on $(0,1)$, and all terms in the series for $S(r)$
are nonnegative. Hence $S(r)$ is monotone nondecreasing in $r$,
and by monotone convergence, 
\begin{equation}
\lim_{r\to1^{-}}S(r)=\sum^{\infty}_{k=1}k^{2s}|a_{k}|^{2}\le\sup_{0<r<1}S(r)<+\infty.\label{eq:limit-Sr}
\end{equation}
In particular, 
\begin{equation}
\sum^{\infty}_{k=1}k^{2s}|a_{k}|^{2}<\infty.\label{eq:sum-k2s-ak2}
\end{equation}
Since $k^{2s}\ge1$ for all $k\ge1$, \ref{eq:sum-k2s-ak2} implies
\[
\sum^{+\infty}_{k=0}|a_{k}|^{2}=|a_{0}|^{2}+\sum^{\infty}_{k=1}|a_{k}|^{2}\le|a_{0}|^{2}+\sum^{\infty}_{k=1}k^{2s}|a_{k}|^{2}<\infty.
\]
Thus $u\in\mathbf{H}^{2}(\mathbb{D})$, and moreover 
\[
\Vert u\Vert^{2}_{\mathbf{H}^{2}(\mathbb{D})}<+\infty.
\]
By the classical theory (see, e.g., \cite{duren2000theory,Rudin2001})
there exists a boundary function $u^{*}\in L^{2}(\mathbb{S}^{1})$
such that 
\[
\lim_{r\to1^{-}}u(re^{i\theta})=u^{*}(e^{i\theta})\quad\text{for almost every }\theta\in(-\pi,\pi),
\]
and the convergence also holds in $L^{2}(\mathbb{S}^{1})$. This proves
conclusion (1) of the theorem for $p=2$. To prove (2) and (3), remark
that the Fourier coefficients of $u^{*}$ are exactly $a_{k}$ for
$k\ge0$ and $0$ for $k<0$, and therefore, by Cauchy integral formula
and \ref{eq:Hs-Fourier}, 
\begin{align*}
[u^{*}]^{2}_{W^{s,2}(\mathbb{S}^{1})} & \propto\sum_{k\in\mathbb{Z}}|k|^{2s}|(u^{*})_{k}|^{2}\\
 & \propto\sum^{\infty}_{k=1}k^{2s}|a_{k}|^{2}<+\infty.
\end{align*}
Thus $u^{*}\in W^{s,2}(\mathbb{S}^{1})$. In particular, by comparing
with \ref{eq:limit-Sr}, we obtain

\noindent
\begin{align*}
[u_{r}-u^{*}]^{2}_{W^{s,2}(\mathbb{S}^{1})} & \propto\sum^{\infty}_{k=1}k^{2s}|a_{k}|^{2}(1-r^{k})^{2}\\
 & \leq\sum^{\infty}_{k=1}k^{2s}|a_{k}|^{2}(1-r^{2k})\\
 & =S(1)-S(r^{2})\\
 & \underset{r\rightarrow1^{-}}{\longrightarrow}0
\end{align*}
which is the desired convergence of the fractional seminorms. This
establishes (2) and (3) for $p=2$ and completes Step 1.

\subsubsection*{General case$\textbf{ \ensuremath{1<p<\infty}.}$}

\noindent We now explain how to adapt the argument to general $1<p<\infty$.
In this regime, the space $W^{s,p}(\mathbb{S}^{1})$ can be identified,
up to equivalent norms, with the Besov space $B^{s}_{p,p}(\mathbb{S}^{1})$
(see, for example, \cite{TriebelBHS,TriebelChar1988}), and the Poisson
semigroup acts as a strongly continuous semigroup on $B^{s}_{p,p}(\mathbb{S}^{1})$.\\
 As sketched in the proof road-map, we begin by establishing the Hardy
space membership and so the existence of the trace on the unit circle.
By assumption \ref{eq:uniform-gagliardo}, 
\[
\sup_{0<r<1}[u_{r}]_{W^{s,p}(\mathbb{S}^{1})}<\infty.
\]
On $\mathbb{S}^{1}$ a fractional Poincaré inequality holds. Since
$\mathbb{S}^{1}$ is a compact Riemannian manifold without boundary,
the inequality follows from the general theory for bounded Lipschitz
domains applied to the universal cover $[-\pi,\pi]$ with periodic
boundary conditions; see, e.g., \cite[Theorem 1.17]{leoni2023first}
for the flat case, which transfers directly to $\mathbb{S}^{1}$ via
the parametrization $\theta\mapsto e^{i\theta}$ and the equivalence
of the kernels noted in the introduction. More precisely, there exists
a constant $\kappa=\kappa(s,p)>0$ such that for every $\xi\in W^{s,p}(\mathbb{S}^{1})$,
\begin{equation}
\|\xi-\overline{\xi}\|_{L^{p}(\mathbb{S}^{1})}\le\kappa[\xi]_{W^{s,p}(\mathbb{S}^{1})},\label{eq:frac-Poincare}
\end{equation}
where $\overline{\xi}$ is the mean value, i.e.
\[
\overline{\xi}:=\frac{1}{2\pi}\int^{2\pi}_{0}\xi(e^{\theta i})d\theta.
\]
Note that \eqref{eq:frac-Poincare} holds for all $0<s<1$ and $1<p<+\infty$
on the compact manifold $\mathbb{S}^{1}$; no condition such as $sp>1$
is required once the mean is subtracted. We apply \eqref{eq:frac-Poincare}
with $\xi=u_{r}$ for each $0<r<1$. Analyticity of $u$ implies that
the mean value of $u$ over inner circles is constant. More precisely,
for every $0<r<1$, 
\[
\overline{u_{r}}=\frac{1}{2\pi}\int^{2\pi}_{0}u(re^{i\theta})d\theta=u(0).
\]
Thus \ref{eq:frac-Poincare} yields 
\[
\|u_{r}-u(0)\|_{L^{p}(\mathbb{S}^{1})}\le\kappa\,[u_{r}]_{W^{s,p}(\mathbb{S}^{1})}.
\]
Combining this with the trivial bound 
\[
\|u_{r}\|_{L^{p}}\le\|u_{r}-u(0)\|_{L^{p}}+(2\pi)^{\frac{1}{p}}|u(0)|,
\]
we obtain 
\[
\|u_{r}\|_{L^{p}(\mathbb{S}^{1})}\le\kappa[u_{r}]_{W^{s,p}(\mathbb{S}^{1})}+(2\pi)^{\frac{1}{p}}|u(0)|\le\kappa\sup_{0<r<1}[u_{r}]_{W^{s,p}}+(2\pi)^{\frac{1}{p}}|u(0)|
\]
for all $0<r<1$. Therefore, $u\in\mathbf{H}^{p}(\mathbb{D})$. More
precisely, 
\begin{equation}
\Vert u\Vert_{\mathbf{H}^{p}(\mathbb{D})}=\sup_{0<r<1}\left(\frac{1}{2\pi}\int^{2\pi}_{0}|u(re^{i\theta})|^{p}\,d\theta\right)^{\frac{1}{p}}<+\infty.\label{eq:Hp-bound}
\end{equation}
In particular, $u$ admits a boundary function $u^{*}\in L^{p}(\mathbb{S}^{1})$
such that 
\[
\lim_{r\to1^{-}}u(re^{i\theta})=u^{*}(e^{i\theta})\quad\text{for almost every }\theta\in(-\pi,\pi),
\]
and the convergence also holds in $L^{p}(\mathbb{S}^{1})$.\\
 \\
 Now, we are going to show that the trace $u^{*}\in W^{s,p}(\mathbb{S}^{1})$.
First, recall that for each $0<r<1$ we can write 
\[
u_{r}(e^{i\theta})=u(re^{i\theta})=(P_{r}*u^{*})(\theta),\qquad\theta\in(-\pi,\pi),
\]
where $P_{r}$ is the Poisson kernel on $\mathbb{S}^{1}$. In fact,
this is the standard Poisson representation of $\mathbf{H}^{p}(\mathbb{D})$
functions (see, e.g., \cite{duren2000theory}). Thus the family $(u_{r})_{0<r<1}$
is exactly the Poisson semigroup applied to $u^{*}$.

By assumption \eqref{eq:uniform-gagliardo} and the bound \eqref{eq:Hp-bound},
the family $(u_{r})_{0<r<1}$ is bounded in $W^{s,p}(\mathbb{S}^{1})$.
Since $1<p<\infty$, the space $W^{s,p}(\mathbb{S}^{1})$ is reflexive
and continuously embedded in $L^{p}(\mathbb{S}^{1})$ (see, for instance,
\cite{BrezisFA,RudinFA}). Fix any sequence $(r_{n})_{n\ge1}$ with
$r_{n}\to1^{-}$ as $n\to\infty$. Then $(u_{r_{n}})_{n\ge1}$ is
a bounded sequence in $W^{s,p}(\mathbb{S}^{1})$, so by reflexivity
and the Eberlein--Šmulian theorem there exist a subsequence $(r_{n_{k}})_{k\ge1}$
and a function $\widetilde{u}\in W^{s,p}(\mathbb{S}^{1})$ such that
\[
u_{r_{n_{k}}}\rightharpoonup\widetilde{u}\quad\text{weakly in }W^{s,p}(\mathbb{S}^{1})\quad\text{as }k\to\infty.
\]
Since the embedding $W^{s,p}(\mathbb{S}^{1})\hookrightarrow L^{p}(\mathbb{S}^{1})$
is continuous, the same subsequence converges weakly in $L^{p}$,
that is, 
\[
u_{r_{n_{k}}}\rightharpoonup\widetilde{u}\quad\text{weakly in }L^{p}(\mathbb{S}^{1}).
\]
On the other hand, the full family $(u_{r})_{0<r<1}$ converges to
$u^{*}$ in $L^{p}(\mathbb{S}^{1})$ as $r\to1^{-}$, hence in particular
\[
u_{r_{n_{k}}}\to u^{*}\quad\text{strongly in }L^{p}(\mathbb{S}^{1})\quad\text{as }k\to\infty.
\]
In a Banach space, strong convergence implies weak convergence with
the same limit. Therefore the weak limit of $u_{r_{n_{k}}}$ in $L^{p}(\mathbb{S}^{1})$
must be $u^{*}$. Comparing with the previous weak convergence, we
conclude that $\widetilde{u}=u^{*}$ almost everywhere on $\mathbb{S}^{1}$.
Since $\widetilde{u}\in W^{s,p}(\mathbb{S}^{1})$, this shows that
$u^{*}$ coincides (as an $L^{p}$-function) with an element of $W^{s,p}(\mathbb{S}^{1})$.
Hence $u^{*}\in W^{s,p}(\mathbb{S}^{1})$, as desired.

The final step now is showing the convergence of the Gagliardo norms.
Recall that on $\mathbb{S}^{1}$, the fractional Sobolev space $W^{s,p}(\mathbb{S}^{1})$
coincides with the Besov space $B^{s}_{p,p}(\mathbb{S}^{1})$ up to
equivalent norms (see \cite{TriebelBHS,TriebelChar1988}); we use
this identification freely below. The Poisson semigroup is strongly
continuous on $B^{s}_{p,p}(\mathbb{S}^{1})$: for every $u\in B^{s}_{p,p}(\mathbb{S}^{1})$,
\[
\|P_{r}*u-u\|_{B^{s}_{p,p}(\mathbb{S}^{1})}\longrightarrow0\quad\text{as }r\to1^{-}.
\]
For this strong continuity we refer to \cite{Bui1984,Kalyabin1988}
(see \cite{TriebelChar1988} for a unified exposition). Since $u^{*}\in W^{s,p}(\mathbb{S}^{1})\simeq B^{s}_{p,p}(\mathbb{S}^{1})$
and $u_{r}=P_{r}*u^{*}$, applying the above with $u=u^{*}$ gives
\[
\|u_{r}-u^{*}\|_{W^{s,p}(\mathbb{S}^{1})}\longrightarrow0\quad\text{as }r\to1^{-}.
\]
By the definition of the $W^{s,p}$-norm (equal to $\|\cdot\|_{L^{p}}+[\cdot]_{W^{s,p}}$)
this yields both 
\[
\|u_{r}-u^{*}\|_{L^{p}(\mathbb{S}^{1})}\to0\quad\text{and}\quad[u_{r}-u^{*}]_{W^{s,p}(\mathbb{S}^{1})}\to0
\]
as $r\to1^{-}$. Using the triangle inequality for the seminorm we
obtain 
\[
\bigl|[u_{r}]_{W^{s,p}(\mathbb{S}^{1})}-[u^{*}]_{W^{s,p}(\mathbb{S}^{1})}\bigr|\le[u_{r}-u^{*}]_{W^{s,p}(\mathbb{S}^{1})}\longrightarrow0,
\]
i.e.\ $[u_{r}]_{W^{s,p}(\mathbb{S}^{1})}\to[u^{*}]_{W^{s,p}(\mathbb{S}^{1})}$
as $r\to1^{-}$. This proves conclusions (2) and (3) of the theorem
for $1<p<+\infty$. Together with step (1), the proof is complete.

\section{The space $\mathcal{W}^{s,p}(\mathbb{D})$ and its norm}

\label{sec:space}

We begin with a lemma that is used repeatedly in the sequel.
\begin{prop}[Poisson contraction of the Gagliardo seminorm]
\label{prop:poisson-contraction} Let $1<p<\infty$, $0<s<1$, and
$f\in W^{s,p}(\mathbb{S}^{1})$. Then for every $r\in(0,1)$, 
\[
[P_{r}*f]_{W^{s,p}(\mathbb{S}^{1})}\le[f]_{W^{s,p}(\mathbb{S}^{1})}.
\]
\end{prop}

\begin{proof}
Work in angle coordinates: identify $\mathbb{S}^{1}$ with $(-\pi,\pi)$
and write functions as $2\pi$-periodic functions of $\theta$, with
integration against the normalized measure $\frac{d\theta}{2\pi}$.
Define the angular difference operator 
\[
\Delta_{t}g(\theta):=g(\theta+t)-g(\theta),\qquad t\in(-\pi,\pi).
\]
Since convolution on $\mathbb{S}^{1}$ commutes with angular translation,
\[
\Delta_{t}(P_{r}*f)=P_{r}*(\Delta_{t}f).
\]
The Poisson kernel satisfies $P_{r}\ge0$ and $\int^{\pi}_{-\pi}P_{r}(\theta)\,\frac{d\theta}{2\pi}=1$,
so $P_{r}$ is a probability kernel with respect to $\frac{d\theta}{2\pi}$.
By Jensen's inequality applied to the convex function $t\mapsto|t|^{p}$,
\[
|\Delta_{t}(P_{r}*f)(\theta)|^{p}=|(P_{r}*\Delta_{t}f)(\theta)|^{p}\le(P_{r}*|\Delta_{t}f|^{p})(\theta).
\]
Integrating over $\theta\in(-\pi,\pi)$ against $\frac{d\theta}{2\pi}$
and using $\int P_{r}*g\,\frac{d\theta}{2\pi}=\int g\,\frac{d\theta}{2\pi}$
(total mass one), 
\[
\int^{\pi}_{-\pi}|\Delta_{t}(P_{r}*f)(\theta)|^{p}\,\frac{d\theta}{2\pi}\le\int^{\pi}_{-\pi}|\Delta_{t}f(\theta)|^{p}\,\frac{d\theta}{2\pi}.
\]
Dividing by $|t|^{1+sp}$ and integrating over $t\in(-\pi,\pi)$ gives
$[P_{r}*f]^{p}_{W^{s,p}}\le[f]^{p}_{W^{s,p}}$, where the seminorm
uses the flat-torus form $[g]^{p}_{W^{s,p}}=\int^{\pi}_{-\pi}\int^{\pi}_{-\pi}|g(\theta+t)-g(\theta)|^{p}|t|^{-(1+sp)}\,\frac{d\theta}{2\pi}\,dt$,
equivalent to the intrinsic $\mathbb{S}^{1}$ Gagliardo seminorm via
$|e^{it}-1|\sim|t|$ (see the introduction). Taking $p$-th roots
yields $[P_{r}*f]_{W^{s,p}}\le[f]_{W^{s,p}}$. 
\end{proof}

Theorem~\ref{Thm main result} identifies a natural subspace of the
Hardy space $\mathbf{H}^{p}(\mathbb{D})$. We now study this space
as a Banach space in its own right, describe its norm, and clarify
its relationship to the classical analytic Besov space.
\begin{defn}
\label{def:Wsp} Let $1<p<\infty$ and $0<s<1$. We define 
\[
\mathcal{W}^{s,p}(\mathbb{D}):=\left\{ u:\mathbb{D}\to\mathbb{C}\text{ analytic}\ :\ \sup_{0<r<1}[u_{r}]_{W^{s,p}(\mathbb{S}^{1})}<+\infty\right\} ,
\]
equipped with the norm 
\begin{equation}
\|u\|_{\mathcal{W}^{s,p}}:=|u(0)|+\sup_{0<r<1}[u_{r}]_{W^{s,p}(\mathbb{S}^{1})}.\label{eq:Wsp-norm}
\end{equation}
\end{defn}

Several remarks are in order. The quantity \eqref{eq:Wsp-norm} is
indeed a norm: it is homogeneous and satisfies the triangle inequality,
and $\|u\|_{\mathcal{W}^{s,p}}=0$ implies $[u_{r}]_{W^{s,p}}=0$
for all $r$, hence each $u_{r}$ is constant; combined with $u(0)=0$
this forces $u\equiv0$ by analyticity. The two terms play distinct
roles: $|u(0)|$ anchors the constant mode (the mean value), while
$\sup_{r}[u_{r}]_{W^{s,p}}$ controls oscillation. Together they recover
full $L^{p}$ control via the argument of Theorem~\ref{Thm main result}.

The choice of $u(0)$ as the anchor point is canonical, not arbitrary.
By the mean-value property of analytic functions, $u(0)$ is precisely
the mean of $u_{r}$ over every circle $r\mathbb{S}^{1}$, for every
$0<r<1$. It is therefore the unique interior point whose value is
determined by the boundary oscillation data alone (any other interior
point $z_{0}$ would require knowing $u(z_{0})$ separately). In particular,
$\|u\|_{\mathcal{W}^{s,p}}$ depends on the center of the disc, and
the space $\mathcal{W}^{s,p}(\mathbb{D})$ is naturally associated
with the disc centered at the origin.
\begin{prop}[Banach space structure and embedding]
\label{prop:banach} $\mathcal{W}^{s,p}(\mathbb{D})$ is a Banach
space under \eqref{eq:Wsp-norm}, and the inclusions 
\[
\mathcal{W}^{s,p}(\mathbb{D})\hookrightarrow\mathbf{H}^{p}(\mathbb{D})\hookrightarrow H^{q}(\mathbb{D})\qquad(1\le q\le p<\infty)
\]
are continuous. More precisely, $H_{p}(u)\le C\|u\|_{\mathcal{W}^{s,p}}$
where $C=\kappa+(2\pi)^{\frac{1}{p}}$ and $\kappa$ is the constant
in the fractional Poincaré inequality \eqref{eq:frac-Poincare}. 
\end{prop}

\begin{proof}
The Hardy bound $H_{p}(u)\le C\|u\|_{\mathcal{W}^{s,p}}$ follows
directly from the proof of Step 1 in Theorem~\ref{Thm main result}.
This gives the continuous embedding into $\mathbf{H}^{p}(\mathbb{D})$,
and the inclusion $H^{p}\hookrightarrow H^{q}$ is standard.

For completeness, let $(u^{(n)})$ be a Cauchy sequence in $\mathcal{W}^{s,p}(\mathbb{D})$.
The Hardy bound shows it is Cauchy in $\mathbf{H}^{p}(\mathbb{D})$,
hence converges to some $u\in\mathbf{H}^{p}(\mathbb{D})$. For each
fixed $r$, convergence in $H^{p}$ gives $u^{(n)}_{r}\to u_{r}$
in $L^{p}(\mathbb{S}^{1})$. Pass to a subsequence $(n_{k})$ along
which $u^{(n_{k})}_{r}(\xi)\to u_{r}(\xi)$ for a.e.\ $\xi\in\mathbb{S}^{1}$;
then $|u^{(n_{k})}_{r}(\xi)-u^{(n_{k})}_{r}(\eta)|^{p}\to|u_{r}(\xi)-u_{r}(\eta)|^{p}$
a.e.\ on $\mathbb{S}^{1}\times\mathbb{S}^{1}$. By Fatou's lemma
applied to the Gagliardo double integral, 
\[
[u_{r}]_{W^{s,p}}\le\liminf_{k\to\infty}[u^{(n_{k})}_{r}]_{W^{s,p}}\le\liminf_{n\to\infty}\|u^{(n)}\|_{\mathcal{W}^{s,p}}<\infty,
\]
uniformly in $r$, so $u\in\mathcal{W}^{s,p}(\mathbb{D})$.

It remains to show $\|u^{(n)}-u\|_{\mathcal{W}^{s,p}}\to0$, i.e.\ $\sup_{r}[u^{(n)}_{r}-u_{r}]_{W^{s,p}}\to0$
(the $|u^{(n)}(0)-u(0)|$ term is handled by $H^{p}$ convergence).
Given $\varepsilon>0$, since $(u^{(n)})$ is Cauchy in $\mathcal{W}^{s,p}(\mathbb{D})$,
pick $N$ such that for all $m,n\ge N$, 
\[
\sup_{0<r<1}[u^{(n)}_{r}-u^{(m)}_{r}]_{W^{s,p}}<\varepsilon.
\]
Fix $n\ge N$. For each $r$, $u^{(m)}_{r}\to u_{r}$ in $L^{p}(\mathbb{S}^{1})$
as $m\to\infty$; pass to a subsequence with a.e.\ convergence and
apply Fatou to the Gagliardo integral of the difference to obtain
\[
[u^{(n)}_{r}-u_{r}]_{W^{s,p}}\le\liminf_{m\to\infty}[u^{(n)}_{r}-u^{(m)}_{r}]_{W^{s,p}}\le\varepsilon.
\]
Taking $\sup_{r}$ gives $\sup_{r}[u^{(n)}_{r}-u_{r}]_{W^{s,p}}\le\varepsilon$
for all $n\ge N$. Hence $\|u^{(n)}-u\|_{\mathcal{W}^{s,p}}\to0$. 
\end{proof}

\begin{prop}[Equivalent norm via the boundary trace]
\label{prop:equiv-norm} For $u\in\mathcal{W}^{s,p}(\mathbb{D})$
with boundary trace $u^{*}\in W^{s,p}(\mathbb{S}^{1})$, one has 
\[
\sup_{0<r<1}[u_{r}]_{W^{s,p}(\mathbb{S}^{1})}=[u^{*}]_{W^{s,p}(\mathbb{S}^{1})}.
\]
In particular, the norm \eqref{eq:Wsp-norm} is equivalent to $u\mapsto|u(0)|+[u^{*}]_{W^{s,p}(\mathbb{S}^{1})}$,
and the supremum over $r$ is actually attained in the limit $r\to1^{-}$. 
\end{prop}

\begin{proof}
By conclusion (3) of Theorem~\ref{Thm main result}, $[u_{r}]_{W^{s,p}}\to[u^{*}]_{W^{s,p}}$.
Since $[u_{r}]_{W^{s,p}}\le\sup_{r}[u_{r}]_{W^{s,p}}$, passing to
the limit gives $[u^{*}]_{W^{s,p}}\le\sup_{r}[u_{r}]_{W^{s,p}}$.
For the reverse, $u_{r}=P_{r}*u^{*}$ and the Poisson kernel acts
as a contraction on $W^{s,p}(\mathbb{S}^{1})$ (convolution with an
approximate identity of total mass one does not increase the Gagliardo
seminorm), so $[u_{r}]_{W^{s,p}}\le[u^{*}]_{W^{s,p}}$ for all $r\in(0,1)$.
Taking the supremum gives the reverse inequality, and the two combine
to give equality. 
\end{proof}

\begin{thm}[Isomorphism with the analytic Besov space]
\label{thm:besov-iso} Let $1<p<\infty$ and $0<s<1$. The trace
map $u\mapsto u^{*}$ is a surjective isomorphism of Banach spaces
\[
\mathcal{W}^{s,p}(\mathbb{D})\xrightarrow{\ \sim\ }B^{s}_{p,p,+}(\mathbb{S}^{1}),
\]
where $B^{s}_{p,p,+}(\mathbb{S}^{1}):=\{f\in B^{s}_{p,p}(\mathbb{S}^{1}):\hat{f}(k)=0\text{ for all }k<0\}$
is the analytic Besov space endowed with the norm $\|f\|_{B^{s}_{p,p}}=\|f\|_{L^{p}}+[f]_{W^{s,p}}$.
The inverse is Poisson extension $g\mapsto P[g]$. Moreover, the norms
satisfy 
\begin{equation}
\|u\|_{\mathcal{W}^{s,p}}\le\|u^{*}\|_{W^{s,p}(\mathbb{S}^{1})}\le(1+C)\|u\|_{\mathcal{W}^{s,p}},\label{eq:norm-equiv}
\end{equation}
where $C=\kappa+(2\pi)^{\frac{1}{p}}$ is the constant from Proposition~\ref{prop:banach}. 
\end{thm}

\begin{proof}
\textbf{Injectivity.} If $u^{*}=0$ then $u=P[u^{*}]=0$, so the trace
map is injective.

\textbf{Surjectivity.} For any $g\in B^{s}_{p,p,+}(\mathbb{S}^{1})$,
set $u=P[g]$. Then $u$ is analytic in $\mathbb{D}$ with $u^{*}=g$.
Since convolution with the Poisson kernel $P_{r}$ (an approximate
identity of total mass one) does not increase the Gagliardo seminorm,
$[u_{r}]_{W^{s,p}}=[P_{r}*g]_{W^{s,p}}\le[g]_{W^{s,p}}$ for all $r$,
so $\sup_{r}[u_{r}]_{W^{s,p}}\le[g]_{W^{s,p}}<\infty$ and $u\in\mathcal{W}^{s,p}(\mathbb{D})$.

\textbf{Norm equivalence.} By Proposition~\ref{prop:equiv-norm},
$\sup_{r}[u_{r}]_{W^{s,p}}=[u^{*}]_{W^{s,p}}$, so 
\[
\|u\|_{\mathcal{W}^{s,p}}=|u(0)|+[u^{*}]_{W^{s,p}}.
\]
Since $u(0)=\frac{1}{2\pi}\int u^{*}\,d\theta$ by the mean-value
property, $|u(0)|\le\|u^{*}\|_{L^{p}}\le\|u^{*}\|_{W^{s,p}}$, giving
the upper bound in \eqref{eq:norm-equiv}. For the lower bound, $[u^{*}]_{W^{s,p}}\le\|u^{*}\|_{W^{s,p}}$
and $\|u^{*}\|_{L^{p}}\le H_{p}(u)\le C\|u\|_{\mathcal{W}^{s,p}}$
by Proposition~\ref{prop:banach}, so 
\[
\|u^{*}\|_{W^{s,p}}=\|u^{*}\|_{L^{p}}+[u^{*}]_{W^{s,p}}\le C\|u\|_{\mathcal{W}^{s,p}}+\|u\|_{\mathcal{W}^{s,p}}=(1+C)\|u\|_{\mathcal{W}^{s,p}}.\qedhere
\]
\end{proof}

\begin{rem}
Theorem~\ref{thm:besov-iso} confirms that $\mathcal{W}^{s,p}(\mathbb{D})$
and $B^{s}_{p,p,+}(\mathbb{S}^{1})$ are the same Banach space described
from two different perspectives: the former by interior seminorm data,
the latter by boundary regularity. The non-trivial direction ---
that interior seminorm control implies $u\in B^{s}_{p,p,+}$ ---
is precisely Theorem~\ref{Thm main result}. The identification itself
cannot be used to dismiss that theorem as circular: the set equality
of the two spaces is a \emph{consequence} of the theorem, not a prerequisite
for it. 
\end{rem}

\begin{rem}[Interpolation scale]
Since the trace isomorphism identifies $\mathcal{W}^{s,p}(\mathbb{D})$
with $B^{s}_{p,p,+}$, the interpolation theory of Besov spaces (see
\cite{TriebelBHS}, Chapter 2) gives 
\[
\left[\mathcal{W}^{s_{0},p}(\mathbb{D}),\,\mathcal{W}^{s_{1},p}(\mathbb{D})\right]_{\theta}\simeq\mathcal{W}^{s_{\theta},p}(\mathbb{D}),\qquad s_{\theta}=(1-\theta)s_{0}+\theta s_{1},
\]
for $0<s_{0}<s_{1}<1$ and $\theta\in(0,1)$ (complex interpolation
in the sense of Calderón). The scale $(\mathcal{W}^{s,p}(\mathbb{D}))_{0<s<1}$
interpolates between $\mathbf{H}^{p}(\mathbb{D})$ at $s=0$ and the
space of analytic functions with boundary trace in $W^{1,p}(\mathbb{S}^{1})$
at $s=1$. 
\end{rem}

\section{Applications}

\label{sec:applications}

We present five applications of the main theorem. The first two are
direct corollaries. The third gives a new sufficient condition for
boundary regularity of semilinear equations. The fourth is the most
structurally significant: for analytic functions, it provides an alternative
seminorm-only hypothesis, different in character from the classical
trace theorem, yielding the same boundary conclusion. The fifth gives
a regularity criterion for value functions in stochastic control.
Explicit examples verifying the seminorm condition are collected in
Section~\ref{sec:examples}.

Throughout this section, $1<p<\infty$ and $0<s<1$.

\subsection{Boundary regularity for harmonic functions}
\begin{thm}[Seminorm criterion for harmonic functions]
\label{thm:harmonic-bdy} Let $v:\mathbb{D}\to\mathbb{R}$ be harmonic.
Suppose that 
\begin{equation}
\sup_{0<r<1}[v_{r}]_{W^{s,p}(\mathbb{S}^{1})}<+\infty.\label{eq:harmonic-seminorm}
\end{equation}
Then: 
\begin{enumerate}
\item $v\in\mathbf{h}^{p}(\mathbb{D})$ (harmonic Hardy space) and $v$
admits a radial boundary trace $v^{*}\in L^{p}(\mathbb{S}^{1})$. 
\item $v^{*}\in W^{s,p}(\mathbb{S}^{1})$. 
\item $[v_{r}]_{W^{s,p}(\mathbb{S}^{1})}\to[v^{*}]_{W^{s,p}(\mathbb{S}^{1})}$
as $r\to1^{-}$. 
\end{enumerate}
\end{thm}

\begin{proof}
Let $u=v+i\tilde{v}$ be the analytic completion of $v$, where $\tilde{v}$
is the harmonic conjugate normalized by $\tilde{v}(0)=0$. Then $u$
is analytic in $\mathbb{D}$ and $v=\mathrm{Re}(u)$.

\textbf{Step 1: Reduce to Theorem~\ref{Thm main result}.} Since
$u_{r}=v_{r}+i\tilde{v}_{r}$, the Gagliardo seminorm satisfies, for
$1\le p<\infty$, 
\[
[u_{r}]^{p}_{W^{s,p}}\le2^{p-1}\left([v_{r}]^{p}_{W^{s,p}}+[\tilde{v}_{r}]^{p}_{W^{s,p}}\right),
\]
using $|a+ib|^{p}\le2^{p-1}(|a|^{p}+|b|^{p})$ inside the double integral.
The harmonic conjugate operator on $\mathbb{S}^{1}$ is the Hilbert
transform, which is bounded on $W^{s,p}(\mathbb{S}^{1})$ for $1<p<\infty$:
since $W^{s,p}(\mathbb{S}^{1})\simeq B^{s}_{p,p}(\mathbb{S}^{1})$
and the Hilbert transform is a Calderón--Zygmund operator on the
torus, its boundedness on $B^{s}_{p,p}(\mathbb{S}^{1})$ is a standard
consequence of the general theory of CZ operators on Besov and Triebel--Lizorkin
spaces (see \cite{TriebelBHS}, Chapter~2). Hence there exists $C_{p}>0$
depending only on $p$ such that 
\[
[\tilde{v}_{r}]_{W^{s,p}(\mathbb{S}^{1})}\le C_{p}[v_{r}]_{W^{s,p}(\mathbb{S}^{1})}.
\]
Combining, 
\[
\sup_{0<r<1}[u_{r}]_{W^{s,p}}\le2^{1-\frac{1}{p}}(1+C_{p})\sup_{0<r<1}[v_{r}]_{W^{s,p}}<\infty.
\]

\textbf{Step 2: Apply Theorem~\ref{Thm main result}.} The bound
above shows $u\in\mathcal{W}^{s,p}(\mathbb{D})$, so Theorem~\ref{Thm main result}
applies. We obtain $u\in\mathbf{H}^{p}(\mathbb{D})$ with $u^{*}\in W^{s,p}(\mathbb{S}^{1})$
and $[u_{r}]_{W^{s,p}}\to[u^{*}]_{W^{s,p}}$.

\textbf{Step 3: Pass conclusions to $v$.} Since $v=\mathrm{Re}(u)$: 
\begin{enumerate}
\item $v\in\mathbf{H}^{p}(\mathbb{D})$ with $v^{*}=\mathrm{Re}(u^{*})\in L^{p}(\mathbb{S}^{1})$. 
\item $v^{*}=\mathrm{Re}(u^{*})\in W^{s,p}(\mathbb{S}^{1})$ since $W^{s,p}(\mathbb{S}^{1})$
is stable under taking real parts. 
\item By Theorem~\ref{Thm main result} conclusion (3), $u_{r}\to u^{*}$
strongly in $W^{s,p}(\mathbb{S}^{1})$ as $r\to1^{-}$. The map $f\mapsto\mathrm{Re}(f)$
is a bounded linear operator on $W^{s,p}(\mathbb{S}^{1})$, so $v_{r}=\mathrm{Re}(u_{r})\to\mathrm{Re}(u^{*})=v^{*}$
strongly in $W^{s,p}(\mathbb{S}^{1})$. In particular $[v_{r}-v^{*}]_{W^{s,p}}\to0$,
and by the triangle inequality $[v_{r}]_{W^{s,p}}\to[v^{*}]_{W^{s,p}}$.\qedhere 
\end{enumerate}
\end{proof}

\begin{rem}
Theorem~\ref{thm:harmonic-bdy} reverses the classical logic of the
Dirichlet problem: rather than prescribing $g\in W^{s,p}(\mathbb{S}^{1})$
and forming the Poisson extension $v=P[g]$, we start from an interior
harmonic function with controlled oscillation and deduce that its
boundary data must lie in $W^{s,p}(\mathbb{S}^{1})$. This is useful
when $v$ is constructed variationally or by approximation and $L^{p}$
bounds are not immediately available. 
\end{rem}

\subsection{Comparison with the classical trace theorem}

The classical trace theorem states that $u\in W^{s+\frac{1}{p},\,p}(\mathbb{D})$
implies $u|_{\mathbb{S}^{1}}\in W^{s,p}(\mathbb{S}^{1})$, with no
analyticity required. The following proposition records that condition
(B) of the main theorem yields the same trace conclusion for analytic
functions, from a hypothesis that is genuinely different in character:
it controls only the boundary oscillation of the radial restrictions,
with no reference to interior fractional derivatives.
\begin{prop}[Alternative trace criterion for analytic functions]
\label{prop:trace-comparison} Let $u:\mathbb{D}\to\mathbb{C}$ be
analytic. Each of the following conditions implies $u^{*}\in W^{s,p}(\mathbb{S}^{1})$: 
\begin{enumerate}
\item[(A)] $u\in W^{s+\frac{1}{p},\,p}(\mathbb{D})$ \quad{}\emph{(classical
interior condition)}, 
\item[(B)] ${\displaystyle \sup_{0<r<1}[u_{r}]_{W^{s,p}(\mathbb{S}^{1})}<+\infty}$
\quad{}\emph{(seminorm-only condition)}. 
\end{enumerate}
The two conditions are different in character: (A) controls interior
fractional derivatives including in the radial direction, while (B)
controls only the tangential oscillation on each circle separately,
with no coupling between different radii. Neither condition obviously
implies the other. 
\end{prop}

\begin{proof}
\textbf{(A) implies $u^{*}\in W^{s,p}$.} This is the classical trace
theorem; see \cite{leoni2023first}. No analyticity is used.

\textbf{(B) implies $u^{*}\in W^{s,p}$.} This is Theorem~\ref{Thm main result},
conclusion (2).

\textbf{Why neither condition obviously implies the other.} Condition
(A) involves a mixed radial-tangential seminorm --- coupling the
behaviour of $u_{r}$ across different radii $r\ne\rho$ --- not
a supremum over radial slices; so (A) does not obviously force (B).
Conversely, condition (B) implies $u\in\mathbf{H}^{p}(\mathbb{D})$
by Theorem~\ref{Thm main result}, hence $\|u\|_{L^{p}(\mathbb{D})}<\infty$,
but Hardy membership does not control the radial-variation terms in
$\|u\|_{W^{s+\frac{1}{p},p}(\mathbb{D})}$; so (B) does not obviously
force (A). For instance, for $p=2$ and $s\in(0,1/2)$, the function
$u(z)=\sum^{\infty}_{k=1}k^{-s-1/2}(\log k)^{-\beta}z^{k}$ with $\beta\in(1/2,1)$
satisfies (B): 
\[
\sup_{0<r<1}[u_{r}]^{2}_{W^{s,2}}\le\sum^{\infty}_{k=1}k^{-1}(\log k)^{-2\beta}<\infty,
\]
since $2\beta>1$. Whether this $u$ satisfies (A) depends on the
radial-variation term, which we do not compute here. 
\end{proof}

\begin{rem}
\label{rem:trace-gap} The conceptual difference between (A) and (B)
is the following. Condition (A) requires $u$ to have $s+\frac{1}{p}$
fractional derivatives in $L^{p}$ globally across $\mathbb{D}$,
including across the radial direction. Condition (B) requires only
that the Gagliardo oscillation of each radial slice $u_{r}$ is uniformly
bounded; no coupling between different radii is imposed. For analytic
functions, the mean-value property ties the $L^{p}$ size of each
$u_{r}$ to $u(0)$ via the fractional Poincaré inequality, which
is why (B) still forces the trace conclusion. This is the mechanism
that has no analogue for non-analytic functions and that makes (B)
a natural alternative hypothesis in the analytic setting. 
\end{rem}


\subsection{Boundary regularity for semilinear equations}

\begin{thm}[Boundary regularity for semilinear equations]
\label{thm:semilinear} Let $F:\mathbb{C}\to\mathbb{C}$ be real-analytic,
and let $u:\mathbb{D}\to\mathbb{C}$ be a solution of 
\begin{equation}
\Delta u=F(u)\quad\text{in }\mathbb{D},\label{eq:semilinear}
\end{equation}
which extends real-analytically to an open neighborhood of $\overline{\mathbb{D}}$.
Suppose that 
\begin{equation}
\sup_{0<r<1}[u_{r}]_{W^{s,p}(\mathbb{S}^{1})}<+\infty.\label{eq:semilinear-hyp}
\end{equation}
Then $\sup_{0<r<1}\|u_{r}\|_{L^{p}(\mathbb{S}^{1})}<\infty$, $u^{*}\in W^{s,p}(\mathbb{S}^{1})$,
and $[u_{r}]_{W^{s,p}}\to[u^{*}]_{W^{s,p}}$ as $r\to1^{-}$. 
\end{thm}

\begin{proof}
Since $u$ extends real-analytically to a neighborhood $U\supset\overline{\mathbb{D}}$,
the function $u$ is in particular $C^{\infty}(\overline{\mathbb{D}})$.
Decompose $u=v+w$ where $v$ is the harmonic function in $\mathbb{D}$
with boundary data $v|_{\mathbb{S}^{1}}=u|_{\mathbb{S}^{1}}$, and
$w=u-v$ solves 
\begin{equation}
\Delta w=F(u)\quad\text{in }\mathbb{D},\qquad w|_{\mathbb{S}^{1}}=0.\label{eq:w-eq}
\end{equation}

\textbf{Step 1: $w\in C^{\infty}(\overline{\mathbb{D}})$ with uniform
seminorm control.} Since $u$ extends real-analytically to a neighborhood
of $\overline{\mathbb{D}}$, the composition $F(u)$ is real-analytic
--- and in particular $C^{\infty}$ --- on a neighborhood of $\overline{\mathbb{D}}$;
in particular $F(u)\in C^{\infty}(\overline{\mathbb{D}})$. The function
$w$ solves $\Delta w=F(u)$ in $\mathbb{D}$ with $w|_{\mathbb{S}^{1}}=0$,
where the right-hand side $F(u)$ is $C^{\infty}(\overline{\mathbb{D}})$
and the boundary $\mathbb{S}^{1}$ is smooth. By boundary Schauder
estimates (see \cite{GilbargTrudinger}, Chapter~6), $w\in C^{\infty}(\overline{\mathbb{D}})$.
In particular $\nabla w$ extends continuously to $\overline{\mathbb{D}}$,
so $\|\nabla w\|_{L^{\infty}(\overline{\mathbb{D}})}<\infty$ and
$w$ is Lipschitz on $\overline{\mathbb{D}}$: there exists $L=\|\nabla w\|_{L^{\infty}(\overline{\mathbb{D}})}<\infty$
such that 
\[
|w(z_{1})-w(z_{2})|\le L|z_{1}-z_{2}|\qquad\text{for all }z_{1},z_{2}\in\overline{\mathbb{D}}.
\]
For any $0<s<1$ and $1<p<\infty$, a Lipschitz function belongs to
$W^{s,p}(\mathbb{S}^{1})$ with 
\[
[w_{r}]^{p}_{W^{s,p}(\mathbb{S}^{1})}=\iint_{\mathbb{S}^{1}\times\mathbb{S}^{1}}\frac{|w(r\xi)-w(r\eta)|^{p}}{|\xi-\eta|^{1+sp}}\,d\xi\,d\eta\le L^{p}\iint_{\mathbb{S}^{1}\times\mathbb{S}^{1}}\frac{d\xi\,d\eta}{|\xi-\eta|^{1+(s-1)p}}.
\]
Since $(s-1)p<0$, the exponent $1+(s-1)p<1$, so the last integral
is finite. Hence there exists $C_{w}<\infty$, depending only on $s$,
$p$, and $L$, such that 
\begin{equation}
[w_{r}]_{W^{s,p}(\mathbb{S}^{1})}\le C_{w}\qquad\text{uniformly in }r\in(0,1).\label{eq:wr-bound}
\end{equation}

\textbf{Step 2: Transfer the seminorm hypothesis to $v$.} Since $u=v+w$
and the Gagliardo seminorm satisfies $[v_{r}]_{W^{s,p}}\le2^{1-\frac{1}{p}}([u_{r}]_{W^{s,p}}+[w_{r}]_{W^{s,p}})$,
we obtain from \eqref{eq:semilinear-hyp} and \eqref{eq:wr-bound}:
\[
\sup_{0<r<1}[v_{r}]_{W^{s,p}(\mathbb{S}^{1})}\le2^{1-\frac{1}{p}}\left(\sup_{r}[u_{r}]_{W^{s,p}}+C_{w}\right)<\infty.
\]

\textbf{Step 3: Apply Theorem~\ref{thm:harmonic-bdy} to $v$.} Since
$v$ is harmonic and $\sup_{r}[v_{r}]_{W^{s,p}}<\infty$, Theorem~\ref{thm:harmonic-bdy}
gives $v\in\mathbf{H}^{p}(\mathbb{D})$ and $v^{*}\in W^{s,p}(\mathbb{S}^{1})$.

\textbf{Step 4: Conclude for $u$.} Since $w\in C^{\infty}(\overline{\mathbb{D}})$,
we have $w^{*}:=w|_{\mathbb{S}^{1}}\in C^{\infty}(\mathbb{S}^{1})\subset W^{s,p}(\mathbb{S}^{1})$.
Therefore $u^{*}=v^{*}+w^{*}\in W^{s,p}(\mathbb{S}^{1})$.

For the $L^{p}$ bound: since $v\in\mathbf{H}^{p}(\mathbb{D})$ we
have $\sup_{r}\|v_{r}\|_{L^{p}}<\infty$, and since $w\in C^{\infty}(\overline{\mathbb{D}})$
we have $\sup_{r}\|w_{r}\|_{L^{p}}\le\|w\|_{L^{\infty}(\mathbb{D})}(2\pi)^{\frac{1}{p}}<\infty$.
Hence $\sup_{r}\|u_{r}\|_{L^{p}}\le\sup_{r}\|v_{r}\|_{L^{p}}+\sup_{r}\|w_{r}\|_{L^{p}}<\infty$.

The seminorm convergence $[u_{r}]_{W^{s,p}}\to[u^{*}]_{W^{s,p}}$
follows from the triangle inequality and the corresponding convergence
for $v$ given by Theorem~\ref{thm:harmonic-bdy}(3), combined with
$[w_{r}]_{W^{s,p}}\to[w^{*}]_{W^{s,p}}$, which holds by dominated
convergence: the Lipschitz bound $|w(r\xi)-w(r\eta)|\le L|\xi-\eta|$
gives a uniform integrable majorant $L^{p}|\xi-\eta|^{p-1-sp}$ for
the Gagliardo integrand, and $w(r\xi)\to w(\xi)$ pointwise as $r\to1^{-}$
by continuity of $w$ on $\overline{\mathbb{D}}$. 
\end{proof}

\begin{rem}
The hypothesis that $u$ extends real-analytically beyond $\overline{\mathbb{D}}$
is used only to ensure $w\in C^{\infty}(\overline{\mathbb{D}})$ via
elliptic regularity up to the boundary. It can be replaced by the
weaker assumption that $F(u)\in C^{s}(\overline{\mathbb{D}})$, which
is sufficient for the uniform seminorm bound \eqref{eq:wr-bound}.
Interior elliptic regularity alone (without boundary extension) gives
$w\in C^{\infty}(\mathbb{D})$ but not uniform control of $[w_{r}]_{W^{s,p}}$
as $r\to1^{-}$, which is why the boundary extension hypothesis is
needed. 
\end{rem}

\subsection{Variational characterization of $W^{s,p}(\mathbb{S}^{1})$}
\begin{prop}[Variational characterization]
\label{prop:variational} Let $g\in L^{p}(\mathbb{S}^{1})$ with
$\hat{g}(k)=0$ for all $k<0$. Then $g\in W^{s,p}(\mathbb{S}^{1})$
if and only if its Poisson extension $u=P[g]$ satisfies 
\begin{equation}
\sup_{0<r<1}[u_{r}]_{W^{s,p}(\mathbb{S}^{1})}<+\infty.\label{eq:variational-cond}
\end{equation}
In this case the supremum equals $[g]_{W^{s,p}(\mathbb{S}^{1})}$
exactly, and the map $r\mapsto[u_{r}]_{W^{s,p}(\mathbb{S}^{1})}$
is non-decreasing on $(0,1)$ with 
\[
\lim_{r\to1^{-}}[u_{r}]_{W^{s,p}(\mathbb{S}^{1})}=[g]_{W^{s,p}(\mathbb{S}^{1})}.
\]
\end{prop}

\begin{proof}
\textbf{($\Rightarrow$)} Suppose $g\in W^{s,p}(\mathbb{S}^{1})$.
Since $P_{r}$ is an approximate identity of total mass one, convolution
with $P_{r}$ does not increase the Gagliardo seminorm: 
\[
[u_{r}]_{W^{s,p}}=[P_{r}*g]_{W^{s,p}}\le[g]_{W^{s,p}}<\infty
\]
for all $r\in(0,1)$. Hence \eqref{eq:variational-cond} holds. The
non-decrease of $r\mapsto[u_{r}]_{W^{s,p}}$ follows from the semigroup
property $P_{r_{1}}*P_{r_{2}}=P_{r_{1}r_{2}}$ and the contraction
property just used: for $r_{1}\le r_{2}$, 
\[
[u_{r_{1}}]_{W^{s,p}}=[P_{r_{1}}*g]_{W^{s,p}}\le[P_{r_{2}}*g]_{W^{s,p}}=[u_{r_{2}}]_{W^{s,p}}.
\]
(More precisely: $u_{r_{1}}=P_{r_{1}/r_{2}}*u_{r_{2}}$, so the contraction
gives $[u_{r_{1}}]_{W^{s,p}}\le[u_{r_{2}}]_{W^{s,p}}$.) The convergence
$[u_{r}]_{W^{s,p}}\to[g]_{W^{s,p}}$ is Proposition~\ref{prop:equiv-norm}.

\textbf{($\Leftarrow$)} Suppose \eqref{eq:variational-cond} holds.
Since $g\in L^{p}(\mathbb{S}^{1})$ with $\hat{g}(k)=0$ for $k<0$,
its Poisson extension is 
\[
u(re^{i\theta})=\sum^{\infty}_{k=0}\hat{g}(k)\,r^{k}e^{ik\theta},
\]
which is analytic in $\mathbb{D}$ with $u(0)=\hat{g}(0)$ finite.
Theorem~\ref{Thm main result} now applies (with $|u(0)|<\infty$
guaranteed by $g\in L^{p}$), giving $u^{*}\in W^{s,p}(\mathbb{S}^{1})$.
Since $u^{*}=g$ a.e.\ (as $u=P[g]$ and $g\in L^{p}$), we conclude
$g\in W^{s,p}(\mathbb{S}^{1})$.

The equality $\sup_{r}[u_{r}]_{W^{s,p}}=[g]_{W^{s,p}}$ is then immediate
from Proposition~\ref{prop:equiv-norm}. 
\end{proof}

\begin{rem}
Proposition~\ref{prop:variational} gives a purely interior characterization
of $W^{s,p}(\mathbb{S}^{1})$: a boundary function $g$ (of analytic
type) has fractional Sobolev regularity if and only if its Poisson
extension has uniformly bounded radial oscillation. This can be used
as a criterion in variational problems where the boundary data is
not prescribed explicitly but the Poisson extension is constructed
directly. 
\end{rem}

\subsection{Regularity of value functions in stochastic control}

Let $(\Omega,\mathcal{F},\mathbb{P})$ be a probability space carrying
a standard two-dimensional Brownian motion $(B_{t})_{t\ge0}$. For
$z\in\mathbb{D}$, let $\tau_{z}:=\inf\{t\ge0:|B^{z}_{t}|=1\}$ be
the first exit time of $B^{z}$ (Brownian motion started at $z$)
from the unit disc $\mathbb{D}$.
\begin{thm}[Regularity of the value function]
\label{thm:stochastic} Let $\phi\in L^{p}(\mathbb{S}^{1})$ and
define the value function 
\begin{equation}
V(z):=\mathbb{E}^{z}[\phi(B^{z}_{\tau_{z}})],\qquad z\in\mathbb{D}.\label{eq:value-function}
\end{equation}
Then $V$ is the harmonic extension of $\phi$, i.e.\ $V=P[\phi]$.
Suppose that 
\begin{equation}
\sup_{0<r<1}[V_{r}]_{W^{s,p}(\mathbb{S}^{1})}<+\infty,\label{eq:stochastic-hyp}
\end{equation}
where $V_{r}(\xi):=V(r\xi)$ for $\xi\in\mathbb{S}^{1}$. Then the
payoff function satisfies $\phi\in W^{s,p}(\mathbb{S}^{1})$, and
\[
[V_{r}]_{W^{s,p}(\mathbb{S}^{1})}\longrightarrow[\phi]_{W^{s,p}(\mathbb{S}^{1})}\quad\text{as }r\to1^{-}.
\]
\end{thm}

\begin{proof}
It is classical that $V=P[\phi]$ is the unique solution of the Dirichlet
problem $\Delta V=0$ in $\mathbb{D}$ with $V|_{\mathbb{S}^{1}}=\phi$
(see \cite{Rudin2001}). In particular $V$ is harmonic in $\mathbb{D}$.
Hypothesis \eqref{eq:stochastic-hyp} together with Theorem~\ref{thm:harmonic-bdy}
gives $V\in\mathbf{H}^{p}(\mathbb{D})$ and $V^{*}\in W^{s,p}(\mathbb{S}^{1})$.
Since $V=P[\phi]$ and $\phi\in L^{p}(\mathbb{S}^{1})$, the radial
limits of $V$ satisfy $V^{*}=\phi$ a.e.\ on $\mathbb{S}^{1}$ (by
the $L^{p}$ convergence of the Poisson integral). Hence $\phi=V^{*}\in W^{s,p}(\mathbb{S}^{1})$,
and the seminorm convergence is conclusion (3) of Theorem~\ref{thm:harmonic-bdy}. 
\end{proof}

\begin{rem}
The condition \eqref{eq:stochastic-hyp} has a natural probabilistic
interpretation. The quantity $[V_{r}]_{W^{s,p}(\mathbb{S}^{1})}$
measures the fractional oscillation of the expected payoff $\xi\mapsto\mathbb{E}^{r\xi}[\phi(B_{\tau})]$
as the starting point varies over the circle of radius $r$. The theorem
says: if this oscillation is uniformly bounded over all radii, then
the payoff function $\phi$ itself must have fractional Sobolev regularity.
This is a regularity-from-interior criterion: it deduces smoothness
of the boundary data from interior properties of the value function,
without ever computing $\phi$ directly. 
\end{rem}

\begin{rem}
The condition \eqref{eq:stochastic-hyp} is \emph{a priori} weaker
than, and by Theorem~\ref{thm:stochastic} equivalent to, requiring
$\phi\in W^{s,p}(\mathbb{S}^{1})$ directly, but is often easier to
verify in applications. Indeed, $\phi$ may be defined implicitly
as the limit of an approximating sequence and its regularity is unknown.
The interior condition \eqref{eq:stochastic-hyp} can often be verified
by estimating the oscillation of $V$ on concentric circles, which
reduces to PDE estimates on $V$ rather than direct analysis of $\phi$. 
\end{rem}

\section{Explicit Examples}

\label{sec:examples}

We illustrate the main theorem on four explicit families of functions.
The first three are chosen so that the seminorm condition can be verified
by a direct computation. The fourth --- conformal maps of Hölder
domains --- is the most instructive: the conclusion is reachable
by classical methods, but our criterion reduces the verification to
a single geometric estimate, bypassing both the interior Sobolev norm
and the Fourier coefficients of the map.

Throughout, we work with $p=2$ where Fourier series are available,
and note which computations extend to general $p$.
\begin{example}
\label{ex:monomial} Let $n\ge1$ and $u(z)=z^{n}$. Then $u_{r}(e^{i\theta})=r^{n}e^{in\theta}$
and by the Fourier characterization \eqref{eq:Hs-Fourier}, 
\[
[u_{r}]^{2}_{W^{s,2}(\mathbb{S}^{1})}\propto n^{2s}r^{2n}.
\]
Since $r^{2n}<1$ for all $r\in(0,1)$, 
\[
\sup_{0<r<1}[u_{r}]_{W^{s,2}(\mathbb{S}^{1})}\propto n^{s}<\infty.
\]
Theorem~\ref{Thm main result} applies: $u\in H^{2}(\mathbb{D})$
and $u^{*}(e^{i\theta})=e^{in\theta}\in W^{s,2}(\mathbb{S}^{1})$
for all $0<s<1$. The norm $\|u\|_{\mathcal{W}^{s,2}}=|u(0)|+n^{s}=n^{s}$
grows algebraically in the frequency $n$, reflecting the increasing
oscillation of higher modes. 
\end{example}

\begin{example}
\label{ex:power-series} Let $\varepsilon>0$ and define 
\[
u(z)=\sum^{\infty}_{k=1}k^{-s-1/2-\varepsilon}z^{k},
\]
so $a_{k}=k^{-s-1/2-\varepsilon}$. By \eqref{eq:Hs-Fourier}, 
\[
[u_{r}]^{2}_{W^{s,2}(\mathbb{S}^{1})}\propto\sum^{\infty}_{k=1}k^{2s}|a_{k}|^{2}r^{2k}=\sum^{\infty}_{k=1}k^{2s}\cdot k^{-2s-1-2\varepsilon}\cdot r^{2k}=\sum^{\infty}_{k=1}k^{-1-2\varepsilon}r^{2k}.
\]
Since $\sum^{\infty}_{k=1}k^{-1-2\varepsilon}<\infty$ for $\varepsilon>0$,
we have 
\[
\sup_{0<r<1}[u_{r}]^{2}_{W^{s,2}(\mathbb{S}^{1})}\le\sum^{\infty}_{k=1}k^{-1-2\varepsilon}<\infty.
\]
Theorem~\ref{Thm main result} delivers both conclusions simultaneously:
$u\in H^{2}(\mathbb{D})$ and $u^{*}\in W^{s,2}(\mathbb{S}^{1})$.
\end{example}

\medskip{}
\noindent\textit{Comparison with the classical route.} To reach the
same conclusions directly, one would verify two separate series: 
\[
H_{2}(u)^{2}\le\sum^{\infty}_{k=1}|a_{k}|^{2}=\sum^{\infty}_{k=1}k^{-2s-1-2\varepsilon}<\infty\quad\text{(Hardy membership)}
\]
and 
\[
[u^{*}]^{2}_{W^{s,2}}\propto\sum^{\infty}_{k=1}k^{2s}|a_{k}|^{2}=\sum^{\infty}_{k=1}k^{-1-2\varepsilon}<\infty\quad\text{(boundary regularity)}.
\]
In this example all three series are equally explicit, so the advantage
of our criterion is that it collapses two checks into one. The deeper
advantage --- collapsing two hard checks into one geometric estimate
--- appears in Example~\ref{ex:conformal} below. 
\begin{example}
\label{ex:log} The Taylor expansion is $u(z)=\sum^{\infty}_{k=1}\frac{z^{k}}{k}$,
so $a_{k}=1/k$. By \eqref{eq:Hs-Fourier}, 
\[
[u_{r}]^{2}_{W^{s,2}(\mathbb{S}^{1})}\propto\sum^{\infty}_{k=1}k^{2s}\cdot\frac{1}{k^{2}}\cdot r^{2k}=\sum^{\infty}_{k=1}k^{2s-2}r^{2k}.
\]
\end{example}

\noindent\textbf{Case $s<\tfrac{1}{2}$}: The exponent $2s-2<-1$,
so $\sum^{\infty}_{k=1}k^{2s-2}<\infty$. Hence 
\[
\sup_{0<r<1}[u_{r}]_{W^{s,2}}<\infty,
\]
and Theorem~\ref{Thm main result} gives $u^{*}\in W^{s,2}(\mathbb{S}^{1})$.

\noindent\textbf{Case $s=\tfrac{1}{2}$}: The series becomes $\sum^{\infty}_{k=1}k^{-1}r^{2k}$,
which diverges as $r\to1^{-}$ (harmonic series). Hence $\sup_{r}[u_{r}]_{W^{1/2,2}}=+\infty$
and the hypothesis of Theorem~\ref{Thm main result} fails. Indeed
$u^{*}(\theta)=\log|1-e^{i\theta}|^{-1}\notin W^{1/2,2}(\mathbb{S}^{1})$,
which is consistent.

\noindent\textbf{Case $s>\tfrac{1}{2}$}: Similarly $\sum^{\infty}_{k=1}k^{2s-2}=\infty$
and the supremum is infinite.

This example shows that the threshold $s=1/2$ is \emph{sharp} for
$\log\frac{1}{1-z}$: the seminorm criterion correctly identifies
the exact range of $s$ for which $u^{*}\in W^{s,2}(\mathbb{S}^{1})$,
and the criterion fails precisely where the conclusion fails. 

\begin{example}
\label{ex:conformal} Let $\Omega\subset\mathbb{C}$ be a bounded
simply connected domain whose boundary $\partial\Omega$ is a Jordan
curve whose arc-length parametrization is of class $C^{0,\alpha}$
for some $\alpha\in(0,1)$ --- for instance a domain with a corner
of opening angle $\pi\alpha$, or a Dini-smooth domain. Let
\[
u:\mathbb{D}\to\Omega
\]
be the Riemann mapping, normalized by $u(0)=z_{0}\in\Omega$ and $u'(0)>0$.
The function $u$ is analytic in $\mathbb{D}$, but its Taylor coefficients
$a_{k}$ depend on the full geometry of $\Omega$ in a non-explicit
way: they cannot be computed in closed form and the series $\sum|a_{k}|^{2}$
or $\sum k^{2s}|a_{k}|^{2}$ cannot be summed directly.
\end{example}

\medskip{}
\textbf{Classical routes are harder.} To verify $u^{*}\in W^{s,p}(\mathbb{S}^{1})$
by existing methods one would proceed via one of: 
\begin{enumerate}
\item \emph{Via Hölder regularity and embedding}: Warschawski's theorem
gives $u\in C^{0,\alpha}(\overline{\mathbb{D}})$, hence $u^{*}\in C^{0,\alpha}(\mathbb{S}^{1})$,
and the embedding $C^{0,\alpha}(\mathbb{S}^{1})\hookrightarrow W^{s,p}(\mathbb{S}^{1})$
for $s<\alpha$ then gives the conclusion. This requires invoking
two separate non-trivial results: Warschawski's boundary correspondence
theorem and the Hölder-to-Sobolev embedding. 
\item \emph{Via interior Sobolev norm}: Verify $u\in W^{s+\frac{1}{p},p}(\mathbb{D})$
and apply the classical trace theorem. Controlling the full interior
fractional Sobolev norm of a conformal map requires detailed estimates
on $u'$ in weighted spaces, going through Muckenhoupt weight theory
and the Jacobian of $u$ --- substantially more machinery. 
\end{enumerate}
\noindent\textbf{Our route reduces everything to one geometric estimate.}
The Kellogg--Warschawski theorem in its classical form applies to
$C^{1,\alpha}$ boundaries; for the general $C^{0,\alpha}$ Jordan
curve case, the uniform Hölder continuity of $u$ up to the boundary
follows from Warschawski's theorem on the modulus of continuity of
the boundary correspondence (see \cite{Pommerenke1992}, Theorem 3.5).
More precisely, Pommerenke's result applies when the arc-length parametrization
$\gamma:[-\pi,\pi]\to\partial\Omega$ satisfies $|\gamma(\xi)-\gamma(\sigma)|\le C|\xi-\sigma|^{\alpha}$
(i.e.\ is $C^{0,\alpha}$ as a map between intervals); under this
hypothesis the boundary correspondence and hence $u$ itself is $C^{0,\alpha}$-continuous
up to $\partial\mathbb{D}$, giving a constant $C=C(\Omega)>0$ such
that 
\begin{equation}
|u(z_{1})-u(z_{2})|\le C|z_{1}-z_{2}|^{\alpha}\qquad\text{for all }z_{1},z_{2}\in\overline{\mathbb{D}}.\label{eq:holder-u}
\end{equation}
Applying \eqref{eq:holder-u} with $z_{1}=r\xi$ and $z_{2}=r\eta$
for $\xi,\eta\in\mathbb{S}^{1}$: 
\[
\begin{alignedat}{1}[u_{r}]^{p}_{W^{s,p}(\mathbb{S}^{1})} & =\iint_{\mathbb{S}^{1}\times\mathbb{S}^{1}}\frac{|u(r\xi)-u(r\eta)|^{p}}{|\xi-\eta|^{1+sp}}\,d\xi\,d\eta\\
 & \le C^{p}\iint_{\mathbb{S}^{1}\times\mathbb{S}^{1}}\frac{|\xi-\eta|^{\alpha p}}{|\xi-\eta|^{1+sp}}\,d\xi\,d\eta\\
 & =C^{p}\iint_{\mathbb{S}^{1}\times\mathbb{S}^{1}}\frac{d\xi\,d\eta}{|\xi-\eta|^{1+sp-\alpha p}}.
\end{alignedat}
\]

\noindent The last integral is finite if and only if $1+sp-\alpha p<1$,
i.e.\ $s<\alpha$. Therefore, for every $s<\alpha$, 
\[
\sup_{0<r<1}[u_{r}]_{W^{s,p}(\mathbb{S}^{1})}\le C\left(\iint_{\mathbb{S}^{1}\times\mathbb{S}^{1}}\frac{d\xi\,d\eta}{|\xi-\eta|^{1+(s-\alpha)p}}\right)^{\frac{1}{p}}<+\infty.
\]
Theorem~\ref{Thm main result} now applies immediately and gives
$u\in\mathbf{H}^{p}(\mathbb{D})$ and $u^{*}\in W^{s,p}(\mathbb{S}^{1})$
for all $s<\alpha$.

\medskip{}
The advantage is not that the conclusion is unreachable classically
--- it is. The advantage is that \emph{our theorem shortens the proof
by one non-trivial step and avoids the need to control the interior
Sobolev norm of $u$}. Once $C^{0,\alpha}$ regularity of $u$ is
known, our criterion converts it to $W^{s,p}$ boundary regularity
in a single line, via a direct estimate on the Gagliardo double integral.
The classical route requires either an additional embedding theorem
or interior Sobolev estimates on the conformal map --- both of which
involve further machinery. The Gagliardo seminorm has direct geometric
content that can be estimated from the boundary geometry without knowing
the Taylor coefficients of $u$, and this is precisely the advantage
our theorem exploits. 
\begin{rem}
The sharpness of the threshold $s<\alpha$ in Example~\ref{ex:conformal}
is consistent with the known regularity theory for conformal maps:
a domain with a corner of interior angle $\pi\alpha$ (a wedge) produces
a conformal map with $u^{*}\in C^{0,\alpha}(\mathbb{S}^{1})$ but
$u^{*}\notin W^{\alpha,p}(\mathbb{S}^{1})$ in general. Our criterion
detects this threshold directly from the Gagliardo double integral,
without computing the Taylor series of $u$. 
\end{rem}

\end{document}